\documentclass[12pt]{article}
\usepackage{amsmath,amsthm,amsfonts,amssymb}
\usepackage{fullpage}
\usepackage{multirow}
\usepackage{array}
\usepackage{bbm}
\usepackage[noblocks]{authblk}
\usepackage{verbatim}
\usepackage{tikz}
\usepackage{float}
\usepackage{enumerate}
\usepackage{diagbox}
\usepackage{soul}
\usepackage{url}
\usepackage{mathtools}
\usepackage{hyperref}
\usepackage{listings}

\DeclareMathOperator{\lcm}{lcm}

\newtheorem{theorem}{Theorem}[section]

\newtheorem{lemma}[theorem]{Lemma}

\newtheorem{corollary}[theorem]{Corollary}
\newtheorem{conjecture}[theorem]{Conjecture}

\theoremstyle{definition}

\newtheorem{question}[theorem]{Question}

\newlength{\Oldarrayrulewidth}

\newcommand{\ord}{\operatorname{ord}}

\begin{document}

\author[1]{Chris~Bispels\thanks{cbispel1@umbc.edu}}
\author[2]{Matthew~Cohen\thanks{matthewcohen@cmu.edu}}
\author[3]{Joshua~Harrington\thanks{joshua.harrington@cedarcrest.edu}}
\author[4]{Joshua~Lowrance\thanks{joshua.lowrance@biola.edu}}
\author[5]{Kaelyn~Pontes\thanks{kaelyn.pontes@hastings.edu}}
\author[6]{Leif~Schaumann\thanks{schaumann1@kenyon.edu}}
\author[8]{Tony~W.~H.~Wong\thanks{wong@kutztown.edu}}
\affil[1]{Department of Mathematics, University of Maryland, Baltimore County}
\affil[2]{Department of Mathematical Sciences, Carnegie Mellon University}
\affil[3]{Department of Mathematics, Cedar Crest College}
\affil[4]{Department of Mathematics and Computer Science, Biola University}
\affil[5]{Department of Mathematics, Hastings College}
\affil[6]{Department of Mathematics, Kenyon College}
\affil[7]{Department of Mathematics, Kutztown University of Pennsylvania}
\date{\today}

\title{On Sierpi\'nski and Riesel Repdigits and Repintegers}
\maketitle

\begin{abstract}
For positive integers $b\geq 2$, $k<b$, and $t$, we say that an integer $k_b^{(t)}$ is a $b$-repdigit if $k_b^{(t)}$ can be expressed as the digit $k$ repeated $t$ times in base-$b$ representation, i.e., $k_b^{(t)}
=k(b^t-1)/(b-1)$. In the case of $k=1$, we say that $1_b^{(t)}$ is a $b$-repunit. In this article, we investigate the existsence of $b$-repdigits and $b$-repunits among the sets of Sierpi\'nski numbers and Riesel numbers. A Sierpi\'nski number is defined as an odd integer $k$ for which $k\cdot 2^n+1$ is composite for all positive integers $n$ and Riesel numbers are similarly defined for the expression $k\cdot 2^n-1$.

\end{abstract}

\section{Introduction}

In 1956, Riesel demonstrated that $509203\cdot 2^n-1$ is composite for all positive integers $n$ \cite{riesel}. He further proved the existence of infinitely many odd integers $k$ for which $k\cdot 2^n-1$ is composite for all positive integers $n$. Similarly, in 1960, Sierpi\'nski showed that there are infinitely many odd integers $k$ such that $k\cdot 2^n+1$ is composite for all positive integers $n$ \cite{sierpinski}. These numbers are now known as Riesel numbers and Sierpi\'nski numbers, respectively. To date, 509203 is the smallest known Riesel number. The smallest known Sierpi\'nski number, discovered by John Selfridge in 1962, is 78557. We note that Selfridge never published his discovery.

Building on the foundational work of Riesel and Sierpi\'nski, researchers have explored the presence of Riesel numbers and Sierpi\'nski numbers in various integer sequences. Such investigations have consider Riesel numbers and Sierpi\'nski numbers among binomial coefficients \cite{abbdehsw}, polygonal numbers \cite{be,befsk}, Lucas numbers \cite{bff}, Carmichael numbers \cite{bflps}, perfect $j$-th powers \cite{ffk,fj}, Ruth-Aaron pairs \cite{efk}, Narayana's cow sequence \cite{fh}, and Fibonacci numbers \cite{lh}. 

For positive integers $b\geq 2$, $k<b$, and $t$, we say that an integer $k_b^{(t)}$ is a \emph{$b$-repdigit} if $k_b^{(t)}$ can be expressed as the digit $k$ repeated $t$ times in base-$b$ representation, i.e., $k_b^{(t)}
=k(b^t-1)/(b-1)$. In the case of $k=1$, we say that $1_b^{(t)}$ is a \emph{$b$-repunit}. Repunits and repdigits have been topics of studies in the literature since 1966 \cite{beiler,trigg}. Extending the definition of $b$-repdigits, for any positive integer $k$, we say that an integer $k_b^{(t)}$ is a \emph{$b$-repinteger} if $k_b^{(t)}$ can be expressed as $k$ repeated $t$ times in base-$b$ representation, i.e., $k_b^{(t)}=k(b^{\ell t}-1)/(b^\ell-1)$, where $\ell=\lfloor\log_b(k)\rfloor+1$. In this article, we investigate the existence of Sierpi\'nski numbers and Riesel numbers in $b$-repunits, $b$-repdigits, and $b$-repintegers. 

In Section~\ref{sec:results}, we provide an argument to suggest the nonexistence of $2$-repdigit Sierpi\'nski numbers. In light of this, the following questions are the motivation for the other results in this paper.

\begin{question}\label{question:smallrepunitbase}
What is the smallest integer $\beta_1\geq2$ for which there exists a $\beta_1$-repunit Sierpi\'nski number? 
\end{question}

\begin{question}\label{question:smallrepdigitbase}
What is the smallest integer $\beta_2\geq2$ for which there exists a $\beta_2$-repdigit Sierpi\'nski number?
\end{question}

\begin{question}\label{questions:smallrepintegerbase2}
What is the smallest positive integer $\kappa$ for which there exists a positive integer $t$ such that $\kappa_2^{(t)}$ is a $2$-repinteger Sierpi\'nski number?
\end{question}

We will establish the existence of $\beta_1$, $\beta_2$, and $\kappa$, and show that $\beta_1\leq 147$, $\beta_2\leq 87$, and $\kappa\leq 18107$ by Corollary~\ref{cor:sierpinskismallrepunitbase}, Theorem~\ref{thm:smallrepdigitbase}, and Theorem~\ref{thm:smallrepintegerbase2sierpinski}, respectively. Similar questions can be asked about Riesel numbers.

\begin{question}\label{question:smallrepunitbaseriesel}
What is the smallest integer $\beta'_1\geq2$ for which there exists a $\beta'_1$-repunit Riesel number? 
\end{question}

\begin{question}\label{question:smallrepdigitbaseriesel}
What is the smallest integer $\beta'_2\geq2$ for which there exists a $\beta'_2$-repdigit Riesel number?
\end{question}

\begin{question}\label{questions:smallrepintegerbase2riesel}
What is the smallest positive integer $\kappa'$ for which there exists a positive integer $t$ such that $\kappa'^{(t)}_2$ is a $2$-repinteger Riesel number?
\end{question}

Analogous results for Riesel numbers are $\beta'_1\leq 16518444216571$, $\beta'_2\leq 180$, and $\kappa'\leq 18107$ by Corollary~\ref{cor:rieselsmallrepunitbase}, Theorem~\ref{thm:smallrepdigitbaseriesel}, and Theorem~\ref{thm:smallrepintegerbase2riesel}, respectively.

\section{The Covering System Method}\label{sec:coveringsystems}
Finding Sierpi\'nski numbers and Riesel numbers has historically involved the use of covering systems of the integers. In this section, we demonstrate the method that is commonly used to find these numbers. 

A \emph{covering system of integers}, often referred to simply as a \emph{covering system}, is a finite collection of congruences such that every integer satisfies at least one of the congruences in the collection. One can take $\mathcal{C}_0=\{0\pmod{3}, 1\pmod{3}, 2\pmod{3}\}$ for a simple example of a covering system. The collection $\mathcal{C}_1=\{0\pmod{2}, 0\pmod{3}, 1\pmod{4}, 3\pmod{8}, 11\pmod{12},7\pmod{24}\}$ is a less trivial example of a covering system. We will use this covering system to demonstrate our process for generating Sierpi\'nski numbers and Riesel numbers.

Let $n$ be a fixed positive integer. Since $\mathcal{C}_1$ is a covering system, $n$ must satisfy at least one of the congruences in the covering. Suppose $n\equiv 0\pmod{2}$. Then $n=2w$ for some integer $w$ and
\begin{align*}
k\cdot 2^n+1
&=k\cdot 2^{2w}+1\\
&=k\cdot 4^w+1\\
&\equiv k+1\pmod{3}.
\end{align*}
Thus, if $k\equiv 2\pmod{3}$ and $k>0$, then $k\cdot 2^n+1$ is divisible by $3$ and is therefore composite since $k\cdot 2^n+1>3$ for all positive $n\equiv 0\pmod{2}$. Using this technique, we go through each of the congruences in $\mathcal{C}_1$ to get the following:
\begin{align}\label{eq:sierpinskidemo}
n\equiv 0\pmod{2}\text{ and }k\equiv 2\pmod{3}&\implies k\cdot 2^n+1\equiv 0\pmod{3}\nonumber\\
n\equiv 0\pmod{3}\text{ and }k\equiv 6\pmod{7}&\implies k\cdot 2^n+1\equiv 0\pmod{7}\nonumber\\
n\equiv 1\pmod{4}\text{ and }k\equiv 2\pmod{5}&\implies k\cdot 2^n+1\equiv 0\pmod{5}\nonumber\\
n\equiv 3\pmod{8}\text{ and }k\equiv 2\pmod{17}&\implies k\cdot 2^n+1\equiv 0\pmod{17}\\
n\equiv 11\pmod{12}\text{ and }k\equiv 11\pmod{13}&\implies k\cdot 2^n+1\equiv 0\pmod{13}\nonumber\\
n\equiv 7\pmod{24}\text{ and }k\equiv 32\pmod{241}&\implies k\cdot 2^n+1\equiv 0\pmod{241}.\nonumber
\end{align}
To ensure that $k$ is odd, we further require $k\equiv 1\pmod{2}$. Note that the system of congruences on $k$ has infinitely many solutions by the Chinese remainder theorem. To ensure that $k\cdot 2^n+1$ is not in the set $\{3,7,5,17,13,241\}$, we may select $k$ to be sufficiently large, thus yielding a Sierpi\'nski number. 

It is important to note that each of the implications in \eqref{eq:sierpinskidemo} can be written as 
\begin{equation}\label{eq:sierpinskicongruence}
n\equiv r_j\pmod{m_j}\text{ and }k\equiv -2^{-r_j}\pmod{p_j}\implies k\cdot 2^n+1\equiv 0\pmod{p_j},
\end{equation}
where $r_j\pmod{m_j}$ is a congruence in our covering system and $p_j$ is a prime divisor of $2^{m_j}-1$. We use a similar set of implications to construct Riesel numbers:

\begin{equation}\label{eq:rieselcongruence}
n\equiv r_j\pmod{m_j}\text{ and }k\equiv 2^{-r_j}\pmod{p_j}\implies k\cdot 2^n-1\equiv 0\pmod{p_j}.
\end{equation}

To construct Sierpi\'nski numbers and Riesel numbers from a generic covering system $\mathcal{C}=\{r_j\pmod{m_j}:1\leq j\leq \tau\}$, we want $p_i\neq p_j$ for each $1\leq i<j\leq\tau$. The following theorem of Bang \cite{bang} is a useful tool to help establish this criteria. For a positive integer $m$, we call a prime $p$ a \emph{primitive prime divisor} of $2^m-1$ if $p$ divides $2^m-1$ and $p$ does not divide $2^\mu-1$ for any positive integers $\mu<m$. 

\begin{theorem}[Bang]\label{thm:bang}
For each integer $m\geq 2$ with $m\neq 6$, there exists a primitive prime divisor of $2^m-1$.
\end{theorem}


\section{Main Results}\label{sec:results}
Note that if $k$ is a Sierpi\'nski number found using the covering system method described in Section~\ref{sec:coveringsystems}, then $k\cdot 2^n+1$ is divisible by one of the primes in the set $\{p_1,p_2,\ldots,p_\tau\}$. This observation was known to Erd\H{o}s, who made the following conjecture \cite{guy}.

\begin{conjecture}\label{conj:erdos}
If $k$ is a Sieprinski number, then the smallest prime divisor of $k\cdot 2^n+1$ is bounded as $n$ tends to infinity.
\end{conjecture}

Although this conjecture has not been disproven, there are a few papers that suggest that it may be false \cite{ffk,izotov}. If Conjecture~\ref{conj:erdos} is true, then the following theorem suggests the nonexistence of $2$-repdigit Sierpi\'nski numbers. We note that the argument in the proof of Theorem~\ref{thm:nosierpinskirepdigit} cannot be easily modified to provide insight on the existence of $2$-repdigit Riesel numbers.

\begin{theorem}\label{thm:nosierpinskirepdigit}
The covering system method in Section~$\ref{sec:coveringsystems}$ cannot yield a $2$-repdigit Sierpi\'nski number.
\end{theorem}
\begin{proof}
Let $\mathcal{C}=\{r_j\pmod{m_j}:1\leq j\leq \tau\}$ be a covering system and let $L$ be the least common multiple of the moduli in $\mathcal{C}$. Then $L$ must satisfy at least one of the congruences of $\mathcal{C}$, and it follows that $r_j=0$ for some $1\leq j\leq \tau$. Thus, to satisfy \eqref{eq:sierpinskicongruence}, we need $k\equiv -2^{-r_j}\equiv -1\pmod{p_j}$. Note that if $k$ is a $2$-repdigit, then $k=2^t-1$ for some positive integer $t$. Therefore, we need $2^t-1\equiv -1\pmod{p_j}$, which is impossible since $p_j\neq 2$ for each $1\leq j\leq \tau$.
\end{proof}

While Theorem~\ref{thm:nosierpinskirepdigit} suggests that it is difficult to find $2$-repdigit Sierpi\'nski numbers, the following theorem establishes the existence of $b$-repinteger Sierpi\'nski numbers for any $b\geq2$. A similar proof can establish an analogous result on $b$-repinteger Riesel numbers, as stated in Theorem~\ref{thm:whendigitisriesel}.



\begin{theorem}\label{thm:whendigitissierpinski}
Let $k$ be a Sierpi\'nski number constructed using the covering system method in Section~$\ref{sec:coveringsystems}$. Then for every integer $b\geq 2$, there exist infinitely many positive integers $t$ such that $k_b^{(t)}$ is a Sierpi\'nski number.
\end{theorem}

\begin{theorem}\label{thm:whendigitisriesel}
Let $k$ be a Riesel number constructed using the covering system method in Section~$\ref{sec:coveringsystems}$. Then for every integer $b\geq 2$, there exist infinitely many positive integers $t$ such that $k_b^{(t)}$ is a Riesel number.
\end{theorem}

\begin{proof}[Proof of Theorem~$\ref{thm:whendigitissierpinski}$]
Let $\mathcal{C}=\{r_j\pmod{m_j}\}$ be the covering system that produces $k$ as a Sierpi\'nski number. Further let $p_j$ be a primitive prime divisor of $2^{m_j}-1$ such that $k\equiv-2^{-r_j}\pmod{p_j}$ for each $j$. Let $t$ be a positive integer such that $t\equiv 1\pmod{\text{ord}_{p_j}(b)}$ for each $p_j$ that does not divide $b$, where $\ord_{p_j}(b)$ denotes the order of $b$ modulo $p_j$. We claim that $k_b^{(t)}\equiv k\pmod{p_j}$ for all $j$. Note that the congruence in the claim holds trivially if $p_j$ divides $b$, and if $p_j$ does not divide $b$, then $k_b^{(t)}\equiv k(b^\ell-1)/(b^\ell-1)\equiv k\pmod{p_j}$, where $\ell=\lfloor\log_b(k)\rfloor+1$.

By \eqref{eq:sierpinskicongruence}, when $n\equiv r_j\pmod{m_j}$, $k_b^{(t)}\cdot 2^n+1\equiv k\cdot 2^n+1\equiv 0\pmod{p_j}$. Also, we have $k_b^{(t)}\cdot 2^n+1\geq k\cdot 2^n+1>p_j$. Since $\mathcal{C}$ is a covering system, this ensures that $k_b^{(t)}\cdot 2^n+1$ is composite for all positive integers $n$. Further requiring $t\equiv 1\pmod{2}$ ensures that $k_b^{(t)}$ is an odd integer, and thus a Sierpi\'nski number.
\end{proof}

Besides establishing the existence of infinitely $b$-repinteger Sierpi\'nski numbers for any $b\geq2$, Theorem~\ref{thm:whendigitissierpinski} also provides a bound to the answer to Question~\ref{questions:smallrepintegerbase2} by showing that $\kappa\leq78557$, the smallest known Sierpi\'nski number. We further improve this bound on $\kappa$ via the following theorem.

\begin{theorem}\label{thm:smallrepintegerbase2sierpinski}
Let $t\equiv 25\pmod{56}$. Then $18107_2^{(t)}$ is a $2$-repinteger Sierpi\'nski number.
\end{theorem}
\begin{proof}
Let
$$\{(r_j,m_j,p_j):1\leq j\leq 6\}=\{(0,2,3),(0,3,7),(1,4,5),(3,8,17),(11,12,13),(7,24,241)\}.$$
Notice that $\mathcal{C}=\{r_j\pmod{m_j}:1\leq j\leq 6\}$ is the covering system provided in equation~\eqref{eq:sierpinskidemo}. This establishes that a positive integer $k$ is a Sierpi\'nski number if $k\equiv 8007257\pmod{11184810}$, where $11184810=2\cdot\prod_{j=1}^6p_j$. Hence, it remains to show that $18107_2^{(t)}$ satisfies this congruence.

Notice that $\ell=\lfloor\log_2(18107)\rfloor+1=15$. For $p\in\{3,7,5,17,13,241\}$, we have $2^{15\cdot 56}\equiv 1\pmod{p}$. Therefore, when $t\equiv 25\pmod{56}$, $2^{15\cdot t}\equiv 2^{15\cdot 25}\pmod{p}$. With this, we can check computationally that $18107_2^{(t)}=18107\cdot(2^{15\cdot t}-1)/(2^{15}-1)\equiv 8007257\pmod{p}$ for $p\in\{2,3,7,5,17,13,241\}$, establishing the desired congruence. 
\end{proof}

Analogously, Theorem~\ref{thm:whendigitisriesel} provides a bound $\kappa'\leq509203$, the smallest known Riesel number, and this bound on $\kappa'$ is improved by the following theorem. We present this theorem without a proof, as it will follow as a corollary to Theorems~\ref{thm:smallrepintegerbase2sierpinski} and \ref{thm:additiveinverse}.

\begin{theorem}\label{thm:smallrepintegerbase2riesel}
Let $t\equiv 31\pmod{56}$. Then $18107_2^{(t)}$ is a $2$-repinteger Riesel number.
\end{theorem}


Noting that a $b$-repinteger $k_b^{(t)}=k(b^{\ell t}-1)/(b^\ell-1)$ with $\ell=\lfloor\log_b(k)\rfloor+1$ is a $b^\ell$-repdigit and recalling that $78557$ and $509203$ are the smallest known Sierpi\'nski number and Riesel number, respectively, we have the following corollaries of Theorems~\ref{thm:whendigitissierpinski} and \ref{thm:whendigitisriesel}.

\begin{corollary}\label{cor:sierpinskib}
For all integers $b>78557$, there exist infinitely many $b$-repdigit Sierpi\'nski numbers.
\end{corollary}

\begin{corollary}
For all integers $b>509203$, there exist infinitely many $b$-repdigit Riesel numbers.
\end{corollary}

Corollary~\ref{cor:sierpinskib} provides a bound to the answer to Question~\ref{question:smallrepdigitbase} by showing that $\beta_2\leq78858$, and we significantly improve this bound to $\beta_2\leq87$ via the next theorem. Similarly, we establish in Theorem~\ref{thm:smallrepdigitbaseriesel} that $\beta_2'\leq 180$.

\begin{theorem}\label{thm:smallrepdigitbase}
Let $b\geq 2$ be an integer with $b\equiv 87\pmod{11184810}$ and let $t$ be a positive integer with $t\equiv 59\pmod{120}$. Then $41_b^{(t)}$ is a $b$-repdigit Sierpi\'nski number.
\end{theorem}
\begin{proof}
Let
$$\{(r_j,m_j,p_j):1\leq j\leq 6\}=\{(0,2,3),(2,3,7),(1,4,5),(7,8,17),(7,12,13),(3,24,241)\}.$$
Notice that $\mathcal{C}=\{r_j\pmod{m_j}:1\leq j\leq 6\}$ is a covering system. We claim that $41_b^{(t)}\equiv-2^{-r_j}\pmod{p_j}$ for all $1\leq j\leq6$. Note that the congruence in the claim holds when $j=1$ since $p_1=3$ divides $b$. For $2\leq j\leq6$, the congruence holds since $41(87^{59}-1)/(87-1)\equiv-2^{r_j}\pmod{p_j}$ and $87^{120}\equiv1\pmod{p_j}$. Therefore, $41_b^{(t)}$ is a Sierpi\'nski number by \eqref{eq:sierpinskicongruence}.
\end{proof}

\begin{theorem}\label{thm:smallrepdigitbaseriesel}
Let $b\geq 2$ be an integer with $b\equiv 180\pmod{11184810}$ and let $t$ be a positive integer with $t\equiv 171\pmod{240}$. Then $101_b^{(t)}$ is a $b$-repdigit Riesel number.
\end{theorem}
\begin{proof}
Let
$$\{(r_j,m_j,p_j):1\leq j\leq 6\}=\{(1,2,3),(2,3,7),
(0,4,5),(2,8,17),(10,12,13),(6,24,241)\}.$$
Notice that $\mathcal{C}=\{r_j\pmod{m_j}:1\leq j\leq 6\}$ is a covering system. We claim that $101_b^{(t)}\equiv2^{-r_j}\pmod{p_j}$ for all $1\leq j\leq6$. Note that the congruence in the claim holds when $j\in\{1,3\}$ since $p_1=3$ and $p_3=5$ divide $b$. For $j\in\{2,4,5,6\}$, the congruence holds since $101(180^{171}-1)/(180-1)\equiv 2^{-r_j}\pmod{p_j}$ and $171^{240}\equiv1\pmod{p_j}$. Therefore, $101_b^{(t)}$ is a Riesel number by \eqref{eq:rieselcongruence}.
\end{proof}

The next theorem provides sufficient conditions on $b$ for which there exist $b$-repunit Sierprinski numbers, and its corollary establishes $\beta_1\leq 147$.

\begin{theorem}\label{thm:sierpinskipowersof2}
Let $\tau$ be an integer such that $2^{2^\tau}-1$ has at least two distinct primitive prime divisors. For each $1\leq j\leq\tau$, let $p_j$ be a primitive prime divisor of $2^{2^j}-1$, and let
$$\ell_j=\begin{cases}
1& \text{if }b\equiv 0\pmod{p_j};\\
p_j& \text{if }b\equiv 1\pmod{p_j};\text{ and}\\
\ord_{p_j}(b)& \text{otherwise}.
\end{cases}$$
Furthermore, let $L=\lcm(2,\ell_1,\ell_2,\ldots,\ell_{\tau})$, and let $q$ be a primitive prime divisor of $2^{2^{\tau}}-1$ with $q\neq p_{\tau}$. For any integer $b>2$ with $b\not\equiv0\pmod{q}$, there exist infinitely many $b$-repunit Sierpi\'nski numbers provided that $b$ satisfies one of the following:
\begin{enumerate}[$(i)$]
\item\label{item:bequiv1} $b\equiv 1\pmod{q}$ and $q$ does not divide $L$;
\item\label{item:bnotequiv1} $b\not\equiv 1\pmod{q}$ and there exists an integer $t$ such that $(b^t-1)/(b-1)\equiv -1\pmod{q}$ and $t\equiv 1\pmod{L}$.
\end{enumerate}
\end{theorem}

\begin{proof}
Notice that $\{2^{j-1}\pmod{2^j}:1\leq j\leq \tau\}\cup\{0\pmod{2^{\tau}}\}$ is a covering system. Let $t$ be an integer such that $t\equiv1\pmod{L}$. We claim that $1_b^{(t)}\equiv-2^{-2^{j-1}}\pmod{p_j}$ for all $1\leq j\leq\tau$, which is equivalent to $1_b^{(t)}\equiv1\pmod{p_j}$ since $2^{2^{j-1}}\equiv-1\pmod{p_j}$. Note that the congruence in the claim holds trivially if $b\equiv 0\pmod{p_j}$. If $b\equiv 1\pmod{p_j}$, then $1_b^{(t)}=\sum_{i=0}^{t-1}b^i\equiv t\equiv1\pmod{L}$, implying that $1_b^{(t)}\equiv1\pmod{p_j}$ since $\ell_j=p_j$. Otherwise, $b^t\equiv b\pmod{p_j}$ since $\ell_j=\ord_{p_j}(b)$, which again implies that $1_b^{(t)}\equiv1\pmod{p_j}$.

In view of \eqref{eq:sierpinskicongruence}, for $1_b^{(t)}$ to be a Sierpi\'nski number, it remains to ensure that $1_b^{(t)}\equiv-2^{-0}\equiv-1\pmod{q}$. If $b\equiv1\pmod{q}$, then since $q$ does not divide $L$, we may impose an additional restriction that $t\equiv-1\pmod{q}$. Hence, $1_b^{(t)}=\sum_{i=0}^{t-1}b^i\equiv t\equiv-1\pmod{q}$. If $b\not\equiv1\pmod{q}$, then condition~$(\ref{item:bnotequiv1})$ yields our desired congruence.
\end{proof}


\begin{corollary}\label{cor:sierpinskismallrepunitbase}
Let $b>2$ be such that $b\equiv\mathfrak{b}\pmod{641}$ for some $\mathfrak{b}\in\{1,147,265,378\}$ and $b\not\equiv1\pmod{5}$. Then there are infinitely many $b$-repunit Sierpi\'nski numbers.
\end{corollary}
\begin{proof}
Consider the covering system $\{2^{j-1}\pmod{2^j}:1\leq j\leq 6\}\cup\{0\pmod{2^6}\}$. Let $p_1=3$, $p_2=5$, $p_3=17$, $p_4=257$, $p_5=65537$, $p_6=6700417$, and $q=641$. Further let $L_0=\lcm(2,p_1,p_3,p_4,p_5,p_6,p_1-1,p_2-1,p_3-1,p_4-1,p_5-1,p_6-1)$, which is a multiple of $L$ since $b\not\equiv1\pmod{p_2}$ and $\ord_{p_j}(b)$ divides $p_j-1$ for all $1\leq j\leq6$. Note that $q$ does not divide $L$ since $q$ does not divide $L_0$. Therefore, if $b\equiv1\pmod{q}$, then part~$(\ref{item:bequiv1})$ of Theorem~\ref{thm:sierpinskipowersof2} is satisfied.

Now suppose that $b\not\equiv 1\pmod{q}$. Then $\mathfrak{b}\in\{147,265,278\}$ and $\ord_q(b)=640$. Let
$$t\equiv \begin{cases}
385\pmod{640}& \text{if } \mathfrak{b}=147;\\
513\pmod{640}& \text{if } \mathfrak{b}=265;\\
257\pmod{640}& \text{if } \mathfrak{b}=378\\\end{cases}$$
so that $(b^t-1)/(b-1)\equiv -1\pmod{q}$. Note that $\gcd(640,L_0)=2^7$ and $t\equiv 1\pmod {2^7}$. Hence, we may apply the Chinese remainder theorem to ensure that $t\equiv 1\pmod{L}$. Therefore, part~$(\ref{item:bnotequiv1})$ of Theorem~\ref{thm:sierpinskipowersof2} is satisfied.
\end{proof}

The next theorem is a Riesel analog to Theorem~\ref{thm:sierpinskipowersof2}. Its corollary establishes $\beta_1'\leq 16518444216571$.

\begin{theorem}\label{thm:Rieselpowersof2}
Let $\tau$ be an integer such that $2^{2^\tau}-1$ has at least two distinct primitive prime divisors. Let $p_j$ be a primitive prime divisor of $2^{2^j}-1$ for each $1\leq j\leq\tau$, $q$ be a primitive prime divisor of $2^{2^{\tau}}-1$ with $q\neq p_{\tau}$, and $P=p_1p_2\dotsb p_\tau$. For any integer $b>2$ with $b\equiv1\pmod{P}$, there exist infinitely many $b$-repunit Riesel numbers provided that $b$ satisfies one of the following:
\begin{enumerate}[$(i)$]
\item\label{item:rieselbequiv0} $b\equiv0\pmod{q}$;
\item\label{item:rieselbequiv1} $b\equiv1\pmod{q}$;
\item\label{item:rieselbnotequiv1} $b\not\equiv0\pmod{q}$, $b\not\equiv1\pmod{q}$, and $\gcd(P,\ord_q(b))=1$.
\end{enumerate}
\end{theorem}

\begin{proof}
Notice that $\{2^{j-1}\pmod{2^j}:1\leq j\leq \tau\}\cup\{0\pmod{2^{\tau}}\}$ is a covering system. Let $t$ be an integer such that $t\equiv-1\pmod{P}$. Since $b\equiv1\pmod{P}$, we have $1_b^{(t)}=\sum_{i=0}^{t-1}b^i\equiv t\equiv-1\equiv2^{-2^{j-1}}\pmod{p_j}$ for all $1\leq j\leq\tau$.

In view of \eqref{eq:rieselcongruence}, for $1_b^{(t)}$ to be a Riesel number, it remains to ensure that $1_b^{(t)}\equiv2^{-0}\equiv1\pmod{q}$. Note that this congruence trivially holds if $b\equiv0\pmod{q}$. If $b\equiv1\pmod{q}$, then since $q$ does not divide $P$, we may impose an additional restriction that $t\equiv1\pmod{q}$, which implies that $1_b^{(t)}=\sum_{i=0}^{t-1}b^i\equiv t\equiv1\pmod{q}$. If we are under condition~$(\ref{item:rieselbnotequiv1})$, then we may impose an additional restriction that $t\equiv1\pmod{\ord_q(b)}$. In this case, $1_b^{(t)}=(b^t-1)/(b-1)\equiv(b-1)/(b-1)\equiv1\pmod{p_j}$.
\end{proof}

\begin{corollary}\label{cor:rieselsmallrepunitbase}
Let $b\equiv16518444216571\pmod{18446744073709551615}$. Then there are infinitely many $b$-repunit Riesel numbers.
\end{corollary}

\begin{proof}
Consider the covering system $\{2^{j-1}\pmod{2^j}:1\leq j\leq 6\}\cup\{0\pmod{2^6}\}$. Let $p_1=3$, $p_2=5$, $p_3=17$, $p_4=257$, $p_5=65537$, $p_6=641$, and $q=6700417$. Then $P=p_1p_2\dotsb p_6=2753074036095$, $Pq=18446744073709551615$, and $b\equiv 1\pmod{P}$. Furthermore, $b\not\equiv0\pmod{q}$, $b\not\equiv1\pmod{q}$, and $\gcd(P,\ord_q(b))=\gcd(P,\ord_q(16518444216571))=1$.
\end{proof}

\section{Concluding Remarks}\label{sec:conclusion}

Besides improving the bounds to the answers to Questions~\ref{question:smallrepunitbase} through \ref{questions:smallrepintegerbase2riesel}, there are a few other interesting directions for investigation. For instance, for positive integers $b\geq2$, $k$, and $t$, we may define a \emph{$b$-repstring} as $k_b^{(z;t)}=k(b^{(z+\ell) t}-1)/(b^{z+\ell}-1)$ for any nonnegative integer $z$, where $\ell=\lfloor\log_b(k)\rfloor+1$. Repstrings are a generalization of repintegers since $k_b^{(0;t)}=k_b^{(t)}$, while in general, repstrings allow us to insert $z$ zeroes between repeated occurances of $k$. For example, $1001001$ is a repstring with $b=10$, $k=1$, $t=3$, $z=2$, and $\ell=1$. We may ask the following questions on $b$-repstrings.

\begin{question}\label{questions:smallrepstringbase2}
What is the smallest positive integer $\widetilde{\kappa}$ for which there exists a nonnegative integer $z$ and a positive integer $t$ such that $\widetilde{\kappa}_2^{(z;t)}$ is a $2$-repstring Sierpi\'nski number?
\end{question}

\begin{question}\label{questions:smallrepstringbase2riesel}
What is the smallest positive integer $\widetilde{\kappa}'$ for which there exists a nonnegative integer $z$ and a positive integer $t$ such that $\widetilde{\kappa}'^{(z;t)}_2$ is a $2$-repstring Riesel number?
\end{question}

The following two theorems establish that $\widetilde{\kappa}\leq659$ and $\widetilde{\kappa}'\leq659$.

\begin{theorem}\label{thm:smallrepstringbase2sierpinski}
Let $t\equiv\mathfrak{t}\pmod{2730}$ for some $\mathfrak{t}\in\{131,1361\}$. Then $659_2^{(2;t)}$ is a $2$-repstring Sierpi\'nski number.
\end{theorem}

The proof of Theorem~\ref{thm:smallrepstringbase2sierpinski} resembles that of Theorem~\ref{thm:smallrepintegerbase2sierpinski}. In particular, when $t\equiv131\pmod{2730}$, we can use
$$\{(r_j,m_j,p_j):1\leq j\leq6\}=\{(1,2,3),(2,3,7),(0,4,5),(6,8,17),(10,12,13),(18,24,241)\}$$
to show that $659_2^{(2;t)}=659(2^{12\cdot t}-1)(2^{12}-1)\equiv2131099\pmod{11184810}$ is a Sierpi\'nski number. Note here that $\ell=\lfloor\log_2659\rfloor+1=10$. Similarly, when $t\equiv1361\pmod{2730}$, we can use
$$\{(r_j,m_j,p_j):1\leq j\leq6\}=\{(1,2,3),(1,3,7),(0,4,5),(6,8,17),(2,12,13),(18,24,241)\}$$
to show that $659_2^{(2;t)}\equiv1639459\pmod{11184810}$ is a Sierpi\'nski number.

The next theorem will follow as a corollary to Theorems~\ref{thm:smallrepstringbase2sierpinski} and \ref{thm:additiveinverse}.

\begin{theorem}\label{thm:smallrepstringbase2riesel}
Let $t\equiv\mathfrak{t}\pmod{2730}$ for some $\mathfrak{t}\in\{1369,2599\}$. Then $659_2^{(2;t)}$ is a $2$-repstring Riesel number.
\end{theorem}


We end this article with an interesting observation. From Theorems~\ref{thm:smallrepintegerbase2sierpinski} and \ref{thm:smallrepintegerbase2riesel}, when the repeating integer is $18107$ in base-$2$ representation, the number of repetitions to get a Sierpi\'nski number is $25$ modulo $56$, while the number of repetitions to get a Riesel number is $31$ modulo $56$. Notice that $25$ and $31$ are additive inverses modulo $56$. Similarly, from Theorems~\ref{thm:smallrepstringbase2sierpinski} and \ref{thm:smallrepstringbase2riesel}, when the repeating integer is $659$ in base-$2$ representation, the number of repetitions to get a Sierpi\'nski number and the number of repetitions to get a Riesel number form additive inverses of each other, as $131\equiv-2599\pmod{2730}$ and $1361\equiv-1369\pmod{2730}$. Table~\ref{table:rsinverse}, obtained computationally, establishes this observation in several other cases, and Theorem~\ref{thm:additiveinverse} confirms this observation.

\begin{table}[H]
\begin{tabular}{c|c|c|c}
$k$& $z$& Sierpi\'nski condition& Riesel condition\\
\hline\hline
\multirow{4}{*}{$659$}& \multirow{4}{*}{$2$}& $t\equiv131\pmod{2730}$& $t'\equiv2599\pmod{2730}$\\
& & $k_2^{(z;t)}\equiv2131099\pmod{11184810}$& $k_2^{(z;t')}\equiv762701\pmod{11184810}$\\
\cline{3-4}
& & $t\equiv1361\pmod{2730}$& $t'\equiv1369\pmod{2730}$\\
& & $k_2^{(z;t)}\equiv1639459\pmod{11184810}$& $k_2^{(z;t')}\equiv1254341\pmod{11184810}$\\
\hline
\multirow{4}{*}{$727$}& \multirow{4}{*}{$2$}& $t\equiv127\pmod{2730}$& $t'\equiv2603\pmod{2730}$\\
& & $k_2^{(z;t)}\equiv3098059\pmod{11184810}$& $k_2^{(z;t')}\equiv10702091\pmod{11184810}$\\
\cline{3-4}
& & $t\equiv1507\pmod{2730}$& $t'\equiv1223\pmod{2730}$\\
& & $k_2^{(z;t)}\equiv271129\pmod{11184810}$& $k_2^{(z;t')}\equiv2344211\pmod{11184810}$\\
\hline
\multirow{2}{*}{$1177$}& \multirow{2}{*}{$4$}& $t\equiv19\pmod{84}$& $t'\equiv65\pmod{84}$\\
& & $k_2^{(z;t)}\equiv84319681\pmod{140100870}$& $k_2^{(z;t')}\equiv95997337\pmod{140100870}$\\
\hline
\multirow{4}{*}{$1189$}& \multirow{4}{*}{$1$}& $t\equiv1159\pmod{2730}$& $t'\equiv1571\pmod{2730}$\\
& & $k_2^{(z;t)}\equiv7523281\pmod{11184810}$& $k_2^{(z;t')}\equiv4384979\pmod{11184810}$\\
\cline{3-4}
& & $t\equiv2059\pmod{2730}$& $t'\equiv671\pmod{2730}$\\
& & $k_2^{(z;t)}\equiv7400371\pmod{11184810}$& $k_2^{(z;t')}\equiv4507889\pmod{11184810}$\\
\hline
\multirow{4}{*}{$1549$}& \multirow{4}{*}{$1$}& $t\equiv38\pmod{1365}$& $t'\equiv1327\pmod{1365}$\\
& & $k_2^{(z;t)}\equiv32552687\pmod{209191710}$& $k_2^{(z;t')}\equiv23909173\pmod{209191710}$\\
\cline{3-4}
& & $t\equiv1343\pmod{1365}$& $t'\equiv22\pmod{1365}$\\
& & $k_2^{(z;t)}\equiv198067007\pmod{209191710}$& $k_2^{(z;t')}\equiv67586563\pmod{209191710}$\\
\hline
\multirow{4}{*}{$1747$}& \multirow{4}{*}{$1$}& $t\equiv307\pmod{2730}$& $t'\equiv2423\pmod{2730}$\\
& & $k_2^{(z;t)}\equiv4573999\pmod{11184810}$& $k_2^{(z;t')}\equiv5049251\pmod{11184810}$\\
\cline{3-4}
& & $t\equiv397\pmod{2730}$& $t'\equiv2333\pmod{2730}$\\
& & $k_2^{(z;t)}\equiv7892569\pmod{11184810}$& $k_2^{(z;t')}\equiv1730681\pmod{11184810}$\\
\hline
\multirow{2}{*}{$18107$}& \multirow{2}{*}{$0$}& $t\equiv25\pmod{56}$& $t'\equiv31\pmod{56}$\\
& & $k_2^{(z;t)}\equiv8007257\pmod{11184810}$& $k_2^{(z;t')}\equiv10702091\pmod{11184810}$\\
\hline
\multirow{2}{*}{$26267$}& \multirow{2}{*}{$0$}& $t\equiv9\pmod{56}$& $t'\equiv47\pmod{56}$\\
& & $k_2^{(z;t)}\equiv1624097\pmod{11184810}$& $k_2^{(z;t')}\equiv1730681\pmod{11184810}$\\
\hline
\multirow{2}{*}{$32681$}& \multirow{2}{*}{$0$}& $t\equiv51\pmod{56}$& $t'\equiv5\pmod{56}$\\
& & $k_2^{(z;t)}\equiv4067003\pmod{11184810}$& $k_2^{(z;t')}\equiv6610811\pmod{11184810}$\\
\hline
\end{tabular}
\caption{Sierpi\'nski $2$-repstrings and Riesel $2$-repstrings}\label{table:rsinverse}
\end{table}

Before presenting the theorem, we provide a useful lemma that follows from the work of Filaseta and Harvey \cite{fh}.

\begin{lemma}\label{lem:fh}
Let $a$ be an integer. If $\mathcal{C}=\{r_j\pmod{m_j}\}$ is a covering system, then $\mathcal{C}_a=\{r_j+a\pmod{m_j}\}$ is a covering system.
\end{lemma}

\begin{theorem}\label{thm:additiveinverse}
Let $k$, $\mathfrak{t}$, and $w$ be positive integers, and let $z$ be a nonnegative integer. Then there exists a fixed covering system $\mathcal{C}$ that produces $k_2^{(z;t)}$ as a Sierpi\'{n}ski number for all $t\equiv\mathfrak{t}\pmod{w}$ using the covering system method in Section~$\ref{sec:coveringsystems}$ if and only if there exists a fixed covering system $\mathcal{C}'$ that produces $k_2^{(z;t')}$ as a Riesel number for all $t'\equiv-\mathfrak{t}\pmod{w}$ using the covering system method in Section~$\ref{sec:coveringsystems}$.
\end{theorem}

\begin{proof}
Let $\mathcal{C}=\{r_j\pmod{m_j}\}$ be the covering system that produces $k_2^{(z;t)}$ as a Sierpi\'nski number for all $t\equiv\mathfrak{t}\pmod{w}$. Further let $p_j$ be a primitive prime divisor of $2^{m_j}-1$ such that $k_2^{(z;t)}\equiv-2^{-r_j}\pmod{p_j}$ for each $j$. Note that $\mathcal{C}'=\{r_j+(z+\ell)\mathfrak{t}\pmod{m_j}\}$ is a covering system by Lemma~\ref{lem:fh}, and we claim that $\mathcal{C}'$ produces $k_2^{(z;t')}$ as a Riesel number for all $t'\equiv-\mathfrak{t}\pmod{w}$.

If $2^{z+\ell}\equiv1\pmod{p_j}$ for some $j$, then $kt\equiv k_2^{(z;t)}\equiv-2^{-r_j}\pmod{p_j}$ for all $t\equiv\mathfrak{t}\pmod{w}$. In other words, $k(qw+\mathfrak{t})\equiv-2^{-r_j}\pmod{p_j}$ for all integers $q$. Hence, $k(-qw-\mathfrak{t})\equiv2^{-r_j}\pmod{p_j}$ for all integers $q$, thus $k_2^{(z;t')}\equiv kt'\equiv2^{-r_j}\equiv2^{-(r_j+(z+\ell)\mathfrak{t})}\pmod{p_j}$ for all $t'\equiv-\mathfrak{t}\pmod{w}$.

It remains to show that $k_2^{(z;t')}\equiv2^{-(r_j+(z+\ell)\mathfrak{t})}\pmod{p_j}$ when $2^{z+\ell}\not\equiv1\pmod{p_j}$. Note that $k\not\equiv0\pmod{p_j}$ since $k_2^{(z;t)}$ is a multiple of $k$ and $-2^{-r_j}\not\equiv0\pmod{p_j}$. Further note that $2^{(z+\ell)w}\equiv1\pmod{p_j}$, which can be deduced from the congruence $k_2^{(z;w+\mathfrak{t})}\equiv-2^{-r_j}\equiv k_2^{(z;\mathfrak{t})}\pmod{p_j}$. Hence, $k_2^{(z;t')}\equiv k_2^{(z;w-\mathfrak{t})}\pmod{p_j}$. The proof of the claim is completed as follows:
\begin{align*}
k_2^{(z;\mathfrak{t})}=k\frac{2^{(z+\ell)\mathfrak{t}}-1}{2^{z+\ell}-1}&\equiv-2^{-r_j}\pmod{p_j}\\
k\frac{2^{(z+\ell)\mathfrak{t}}-1}{2^{z+\ell}-1}\cdot\frac{2^{(z+\ell)(w-\mathfrak{t})}-1}{2^{(z+\ell)\mathfrak{t}}-1}&\equiv-2^{-r_j}\cdot\frac{2^{(z+\ell)(w-\mathfrak{t})}-1}{2^{(z+\ell)\mathfrak{t}}-1}\pmod{p_j}\\
k\frac{2^{(z+\ell)(w-\mathfrak{t})}-1}{2^{z+\ell}-1}&\equiv-2^{-r_j}\cdot\frac{2^{(z+\ell)(w-\mathfrak{t})}-1}{2^{(z+\ell)\mathfrak{t}}-1}\cdot\frac{2^{(z+\ell)\mathfrak{t}}}{2^{(z+\ell)\mathfrak{t}}}\pmod{p_j}\\
k_2^{(z;w-\mathfrak{t})}&\equiv-2^{-r_j}\cdot\frac{2^{(z+\ell)w}-2^{(z+\ell)\mathfrak{t}}}{(2^{(z+\ell)\mathfrak{t}}-1)2^{(z+\ell)\mathfrak{t}}}\pmod{p_j}\\
k_2^{(z;w-\mathfrak{t})}&\equiv-2^{-r_j}\cdot\frac{1-2^{(z+\ell)\mathfrak{t}}}{(2^{(z+\ell)\mathfrak{t}}-1)2^{(z+\ell)\mathfrak{t}}}\pmod{p_j}\\
k_2^{(z;w-\mathfrak{t})}&\equiv-2^{-r_j}\cdot(-2^{-(z+\ell)\mathfrak{t}})\pmod{p_j}\\
k_2^{(z;w-\mathfrak{t})}&\equiv2^{-(r_j+(z+\ell)\mathfrak{t})}\pmod{p_j}.
\end{align*}

The proof of the converse follows in a similar fashion.
\end{proof}

\begin{corollary}
The answers $\kappa$ and $\kappa'$ to Questions~$\ref{questions:smallrepintegerbase2}$ and $\ref{questions:smallrepintegerbase2riesel}$, respectively, are equal to each other. Similarly, the answers $\widetilde{\kappa}$ and $\widetilde{\kappa}'$ to Questions~$\ref{questions:smallrepstringbase2}$ and $\ref{questions:smallrepstringbase2riesel}$, respectively, are also equal to each other.
\end{corollary}

\end{document}